\documentclass[11pt,reqno,a4paper]{amsart}
\usepackage{euscript}
\usepackage{comment}

\theoremstyle{plain}
\numberwithin{equation}{section}
\newtheorem{theorem}{Theorem}
\newtheorem{proposition}{Proposition}
\newtheorem{lemma}{Lemma}
\newtheorem{definition}{Definition}

\theoremstyle{remark}
\newtheorem{remark}{Remark}
\newtheorem{example}{Example}
\renewcommand{\epsilon}{\varepsilon}
\renewcommand{\phi}{\varphi}

\DeclareMathOperator{\Ker}{Ker}

\DeclareMathOperator{\codim}{codim}
\def\N{\mathbb{N}}

\def\R{\mathbb{R}}

\def\Id{\text{\rm Id}}
\def\Z{\mathbb{Z}}
\usepackage{xcolor}
\usepackage{amsthm}
\usepackage{amsmath}
\usepackage{amsfonts}

\begin{document}
\title{Dichotomies for triangular systems on Hilbert spaces}

\author{Davor Dragi\v cevi\' c } 
\address{Faculty of Mathematics, University of Rijeka, Croatia}
\email[Davor Dragi\v cevi\' c]{ddragicevic@math.uniri.hr}

\author{Kenneth J. Palmer}
\address{ National Taiwan University, Taipei, Taiwan}
\email[Kenneth J. Palmer]{palmer@math.ntu.edu.tw}

\author{Boris Petkovi\' c}
\address{Faculty of Natural Sciences and Mathematics, University of Banja Luka, Bosnia and Herzegovina}
\email[Boris Petkovi\'c]{boris.petkovic@pmf.unibl.org}

\begin{abstract}
In this article, we study the relationship between the exponential dichotomy properties of a triangular system of linear difference equations and its associated diagonal system on Hilbert spaces. We stress that all previous results in this direction were restricted to the finite-dimensional case. As in the previous work of the first two authors, we rely on the relationship between exponential dichotomies and the so-called admissibility properties. However, this approach requires nontrivial changes when passing from the finite-dimensional to the infinite-dimensional setting.

\vspace{3mm}


\vspace{2mm}

\end{abstract}

\vspace{1cm}
\maketitle

\section{Introduction}
In the present paper, we study properties of nonautonomous linear discrete systems of triangular form given by 
\begin{equation}\label{lde-intro}
x(n+1)=\begin{pmatrix}A_{11}(n) & A_{12}(n)\\ 0 & A_{22}(n)\end{pmatrix}x(n), \quad n\in J',\end{equation}
where $J\in \{\Z, \Z^+, \Z^-\}$ and 
\begin{equation}\label{Jprime}
J':=\begin{cases}
J & J\in \{\Z, \Z^+\}; \\
\Z^-\setminus \{0\} & J=\Z^-.
\end{cases}
\end{equation}
 Here, $A_{11}(n)\in \mathcal L(X_1, X_1)$, $A_{12}(n)\in \mathcal L(X_2, X_1)$ and $A_{22}(n)\in \mathcal L(X_2, X_2)$, where $X_i$, $i=1, 2$ are arbitrary Hilbert spaces and $\mathcal L(X_i, X_j)$ denotes the space of all bounded linear operators from $X_i$ to $X_j$. 

 We are interested in studying the relationship between the existence of an exponential dichotomy for~\eqref{lde-intro} and its associated \emph{diagonal} system given by 
\begin{equation}\label{de-intro}
x(n+1)=\begin{pmatrix}A_{11}(n) & 0\\ 0 & A_{22}(n)\end{pmatrix}x(n), \quad n\in J'.
\end{equation}
This problem has been addressed in a series of recent works~\cite{BFP1, BFP2, BP, DP, CP} in the case when $X_i$, $i=1, 2$ are finite-dimensional spaces. Therefore, the present paper is the first to consider the case where $X_i$, $i=1, 2$ are infinite-dimensional spaces. 

The approach we use is the one initiated in the previous article by the first two authors~\cite{DP}, and is based on the characterization of exponential dichotomies in terms of the so-called admissibility property. Using this method, in~\cite{DP} the authors were able to give shorter proofs of previously known results obtained in~\cite{BFP2, BP}, as well as to obtain several new ones. 

We recall that the concept of admissibility for a general nonautonomous linear difference equation
\begin{equation}\label{dge}
x(n+1)=A(n)x(n), \quad n\in J'
\end{equation}
 relates to the (possibly unique) solvability of a nonhomogeneous difference equation
\[
x(n+1)=A(n)x(n)+y(n), \quad n\in J'
\]
for a suitable class of inputs $(y(n))_{n\in J'}$. We note that the input $(y(n))_{n\in J'}$ and the output $(x(n))_{n\in J}$ need to belong to suitable Banach sequence spaces.  The results that connect the existence of an exponential dichotomy of~\eqref{dge} and the admissibility property have a long history that goes back to the seminal work of Perron~\cite{Pe}, with subsequent important contributions by Massera and Sch\"affer~\cite{MS}, Dalecki\u \i \ and Kre\u \i n~\cite{DK}, Coppel~\cite{Coppel} and Henry~\cite{H} (among many others). We refer to~\cite{BDVbook} for a detailed survey of this theory and its relationship with hyperbolicity.
We stress that the results of Henry~\cite{H, H1}
 are most relevant for our purposes, as they deal with equations~\eqref{dge} on arbitrary Banach spaces without any invertibility assumptions for coefficients $A(n)$ or the uniform boundedness of their norms. 

 In the case $J=\Z$, the characterization of the exponential dichotomy for~\eqref{dge} in terms of admissibility (see Theorem~\ref{Hthm}) has the same form in both the finite-dimensional and infinite-dimensional settings. Consequently, in the case $J=\Z$, we are able to straightforwardly extend to the infinite-dimensional setting the results from~\cite{DP} connecting the existence of exponential dichotomy for~\eqref{lde-intro} and~\eqref{de-intro}.

On the other hand, the corresponding characterization of the exponential dichotomy for~\eqref{dge} when $J\in \{\Z^+, \Z^-\}$ (see Theorems~\ref{Hthm2} and~\ref{Hthm3}) is more complicated in the infinite-dimensional setting, as it requires that certain subspaces of a phase space are closed or complemented. Consequently, for $J\in \{\Z^+, \Z^-\}$ it is a nontrivial task to obtain infinite-dimensional versions of the results from~\cite{DP} especially when it is assumed that \eqref{lde-intro} has a dichotomy. We refer to Remarks~\ref{Remark5}, \ref{Remark7} and~\ref{Remark8} for a detailed description of the difficulties that arise.

The paper is organized as follows: in Section~\ref{Preliminaries} we list relevant results (also indicating certain modifications) dealing with the characterization of the exponential dichotomy for~\eqref{dge} in terms of appropriate admissibility properties. In addition, we introduce the setup of the paper.
In Section~\ref{Results1}, we provide results yielding the existence of an exponential dichotomy for~\eqref{lde-intro} provided that~\eqref{de-intro} admits an exponential dichotomy.  Finally, in Section~\ref{Results2} we provide necessary and sufficient conditions under which the existence of an exponential dichotomy for~\eqref{lde-intro} implies that~\eqref{de-intro} admits an exponential dichotomy. We give separate results that correspond to all possible choices of $J$. Moreover, we treat systems with invertible coefficients.

\section{Preliminaries}\label{Preliminaries}
Let $X=(X, \|\cdot \|)$ be an arbitrary Banach space. 
By $\mathcal L(X)$ we denote the space of all bounded linear operators on $X$ equipped with the operator norm, which we also denote by $\|\cdot \|$.

Take $J\in \{\Z, \Z^+, \Z^-\}$, where $\Z^+$ denotes the set of nonnegative integers and $\Z^-$ denotes the set of nonpositive integers. We recall the notion of an exponential dichotomy.
For a sequence $(A(n))_{n\in J}\subset \mathcal L(X)$, we consider the associated linear discrete system
\begin{equation}\label{lde-f}
    x(n+1)=A(n)x(n), \quad n\in J',
\end{equation}
where  $J'$ is given by~\eqref{Jprime}.
By $\Phi(m,n)$, we will denote the associated cocycle given by 
\[
\Phi(m, n)=\begin{cases}
A(m-1)\cdots A(n) & m>n;\\
\Id & m=n,
\end{cases}
\]
for $m, n\in J$ with $m\ge n$, where $\Id$ denotes the identity operator on $X$.

\begin{definition}
We say that~\eqref{lde-f} admits an \emph{exponential dichotomy} if there exist a family of projections $(P(n))_{n\in J}\subset \mathcal L(X)$ and constants $K, \alpha>0$ such that the following conditions hold:
\begin{itemize}
\item[(i)] for $n\in J'$,
\[
A(n)P(n)=P(n+1)A(n).
\]
Moreover, $A(n)\rvert_{\mathcal N P(n)}\colon \mathcal N P(n)\to \mathcal N P(n+1)$ is invertible, where $\mathcal N P(n)$ denotes the kernel of $P(n)$;
\item[(ii)]  for $m, n\in J$ with $m\ge n$,
\[
\| \Phi(m, n)P(n)\| \le Ke^{-\alpha (m-n)};
\]
\item[(iii)] for $m, n\in J$ with $m\le n$,
\[
\| \Phi(m, n)(\Id-P(n))\| \le Ke^{-\alpha (n-m)},
\]
where $\Phi(m, n):=\left (\Phi(n, m)\rvert_{\mathcal N P(m)}\right )^{-1}\colon \mathcal N P(n)\to \mathcal N P(m)$.
\end{itemize}
\end{definition}

We will need the following classical result obtained by Henry~\cite[Theorem 7.6.5]{H}.
\begin{theorem}\label{Hthm}
Let $J=\Z$. The following statements are equivalent:
\begin{itemize}
\item[(i)] \eqref{lde-f} admits an exponential dichotomy;
\item[(ii)] for each sequence $(y(n))_{n\in J}\subset X$ such that 
$\sup_{n\in J}\|y(n)\|<+\infty$, there exists a unique bounded sequence $(x(n))_{n\in J} \subset X$ satisfying
\begin{equation}\label{ADM-1}
x(n+1)=A(n)x(n)+y(n) \quad n\in J.
\end{equation}
\end{itemize}
\end{theorem}
\begin{remark}
In what follows, each sequence $(y(n))_{n\in J}\subset X$ with the property that $\sup_{n\in J}\|y(n)\|<+\infty$ will be called \emph{bounded sequence}.
\end{remark}
We will also use the versions of Theorem~\ref{Hthm} for cases $J=\Z^+$ and $J= \Z^-$. The following result was obtained by Henry~\cite[Corollary 1.10]{H1}. We recall that a 
subspace $S\subset X$ is \emph{complemented} if there exists a closed subspace $Z\subset X$ such that $X=S\oplus Z$.
\begin{theorem}\label{Hthm2}
    Let $J=\Z^+$. The following statements are equivalent:
  \begin{itemize}
\item[(i)] \eqref{lde-f} admits an exponential dichotomy on $J$;
\item[(ii)] for each bounded sequence $(y(n))_{n\in J}\subset X$, there exists a bounded sequence $(x(n))_{n\in J} \subset X$ satisfying
\begin{equation}\label{ADM2-2}
x(n+1)=A(n)x(n)+y(n), \quad n\in J.
\end{equation}
Moreover, the subspace
\begin{equation}\label{bndsubs}
\mathcal S:=\left \{ v\in X: \ \sup_{n\in \Z^+}\|\Phi(n, 0)v\|<+\infty \right \}
\end{equation}
is closed and complemented. 
\end{itemize}  
\end{theorem}

\begin{remark}\label{120rem}
We note that the statement of Theorem~\ref{Hthm2} can be modified in the following way: in (ii) it is sufficient to assume that $\mathcal S$  in~\eqref{bndsubs} is complemented, i.e. the assumption that $\mathcal S$ is closed can be omitted.  We stress that this is a known fact that follows from the more general results obtained in~\cite{BDV}. However, here we provide a short argument that directly relies on Theorem~\ref{Hthm2}, which is also applicable to dichotomies on $\Z^-$ (which were not discussed in~\cite{BDV}). 

In order to illustrate this, let us assume that for each bounded sequence $(y(n))_{n\in J}\subset X$, there exists a bounded sequence $(x(n))_{n\in J} \subset X$ satisfying~\eqref{ADM2-2}, and that $\mathcal S$ is complemented. Let $Z\subset X$ be a closed subspace such that $X=\mathcal S\oplus Z$. Take an arbitrary bounded sequence 
$(y(n))_{n\in J}\subset X$ and let $(x(n))_{n\in J} \subset X$ be a bounded sequence satisfying~\eqref{ADM2-2}. Write $x(0)$ as $x(0)=v_1+v_2$, where $v_1\in \mathcal S$ and $v_2\in Z$. We define a new sequence $(\bar x(n))_{n\in J}$ by \[ \bar x(n)=x(n)-\Phi(n, 0)v_1, \quad n\in J. \] Then $(\bar x(n))_{n\in J}$ is bounded, satisfies~\eqref{ADM2-2} and $\bar x(0)=v_2\in Z$. Moreover, this is the only sequence with these three properties. In fact, if $(\tilde x(n))_{n\in J}$ is a bounded sequence satisfying~\eqref{ADM2-2} and such that $\tilde x(0)\in Z$, then
\[
\bar x(n)-\tilde x(n)=\Phi(n, 0)(\bar x(0)-\tilde x(0)), \quad n\in J.
\]
Consequently, $\bar x(0)-\tilde x(0)\in \mathcal S\cap Z$, and thus $\bar x(0)=\tilde x(0)$. We conclude that $\bar x(n)=\tilde x(n)$ for $n\in J$, as claimed.

Let $\ell^\infty(J)$ denote the space of all bounded sequences $(x(n))_{n\in J}$ in $X$ equipped with the supremum norm. We define a linear operator $\Gamma \colon \ell^\infty(J)\to \ell^\infty(J)$ by $\Gamma ((y(n))_{n\in J})=(x(n))_{n\in J}$, where $(x(n))_{n\in J}$ is the unique sequence in $\ell^\infty(J)$ that satisfies~\eqref{ADM2-2} and $x(0)\in Z$. It is a simple exercise to show that $\Gamma$ is closed and consequently a bounded linear operator. 

Take $v\in X$. We define two bounded sequences $(x(n))_{n\in J}$ and $(y(n))$ by 
\[
x(n)=\begin{cases}
    v & n=0; \\
    0 & n\neq 0
\end{cases} \quad \text{and} \quad 
y(n)=\begin{cases}
-A(0)v &n=0; \\
0 &  n\neq 0.
\end{cases}
\]
Clearly, \eqref{ADM2-2} holds. Let $(\bar x(n))_{n\in J}:=\Gamma ((y(n))_{n\in J})$. Set $Pv:=v-\bar x(0)\in \mathcal S$. It is easy to verify that $P$ is a  projection onto $\mathcal S$ along $Z$. Moreover, we note that 
\[
\|Pv\| \le \|v\|+\|\bar x(0)\|\le \|v\|+\|\Gamma\| \cdot \|A(0)v\| \le (1+\|\Gamma\| \cdot \|A(0)\|)\|v\|, 
\]
for every $v\in X$. Hence, $P$ is bounded, and consequently $\mathcal S$ is closed.
\end{remark}

The following result is due to Henry~\cite[Corollary 1.11]{H1} and represents the version of Theorem~\ref{Hthm2} for dichotomies on $\Z^-$.
\begin{theorem}\label{Hthm3}
Let $J=\Z^-$. The following statements are equivalent:
\begin{enumerate}
\item[(i)] \eqref{lde-f} admits an exponential dichotomy on $J$;
\item[(ii)] for each bounded sequence $(y(n))_{n\in J'}\subset X$, there exists a bounded sequence $(x(n))_{n\in J}\subset X$ satisfying 
\begin{equation}\label{ADM3-3}
x(n+1)=A(n)x(n)+y(n), \quad n\in J'.
\end{equation}
Moreover, the subspace $\mathcal U$ of $X$ consisting of all $v\in X$ with the property that there exists a bounded sequence $(x(n))_{n\in J}\subset X$ such that $x(0)=v$ and $x(n)=A(n-1)x(n-1)$ for $n\in J$ is closed and complemented. Finally, the only bounded sequence $(x(n))_{n\in J}\subset X$ that satisfies $x(0)=0$ and $x(n)=A(n-1)x(n-1)$ for $n\in J$ is the zero sequence.
\end{enumerate}
\end{theorem}

\begin{remark}\label{1037rem}
Using similar arguments to those in Remark~\ref{120rem}, one can show that in (ii) it is sufficient to assume that $\mathcal U$ is complemented. 
\end{remark}

\subsection{Setup}
Let $X_i=(X_i, \langle \cdot, \cdot\rangle_i)$, $i=1, 2$ be arbitrary Hilbert spaces. For simplicity, we will denote both scalar products $\langle \cdot, \cdot\rangle_1$ and $\langle \cdot, \cdot\rangle_2$ by $\langle \cdot, \cdot \rangle$. Moreover, the norms induced by $\langle \cdot, \cdot \rangle_i$, $i=1, 2$ will be denoted by $\| \cdot \|$.
Then, $X_1\times X_2$ is a Hilbert space with respect to the scalar product
\[
\langle (x_1, x_2), (y_1, y_2)\rangle:=\langle x_1, y_1\rangle+\langle x_2, y_2\rangle.
\]
Instead of considering the norm on $X_1\times X_2$ induced with this scalar product, for us it will be more convenient to work with the equivalent norm given by 
\[
\|(x_1, x_2)\|:=\max \{\|x_1\|, \|x_2\|\}, \quad (x_1, x_2)\in X_1\times X_2.
\]

In this paper, we will consider the system~\eqref{lde-f} with the operators $A(n)$ having the block triangular form
\begin{equation}\label{tf} A (n) =\begin{pmatrix}
A_{11}(n) & A_{12}(n)\\
0 & A_{22}(n)
\end{pmatrix},
\end{equation}
where $A_{11}(n)\in \mathcal L(X_1)$, $A_{22}(n)\in \mathcal L(X_2)$ and $A_{12}(n)\in \mathcal L(X_2, X_1)$, where $\mathcal L(X_2, X_1)$ denotes the space of all bounded linear operators from $X_2$ to $X_1$. We have the corresponding discrete system
\begin{equation}\label{LDE}
x(n+1)=\begin{pmatrix}
A_{11}(n) & A_{12}(n)\\
0 & A_{22}(n)
\end{pmatrix} x(n), \quad n\in J'.
\end{equation}
In this case, to~\eqref{LDE} we associate the corresponding block diagonal system given by 
\begin{equation}\label{ldedd}
    x(n+1)=\begin{pmatrix}
A_{11}(n) & 0\\
0 & A_{22}(n)
    \end{pmatrix} x(n), \quad n\in J'.
\end{equation}
Our goal is to explore the relationship between the existence of an exponential dichotomy for~\eqref{LDE} and~\eqref{ldedd}.
\begin{remark}
Although we focus on systems~\eqref{LDE} of block triangular form which consists of two blocks, using induction, one can extend all the results in our paper to the general case when there are finitely many blocks. 
\end{remark}

We make the following simple observation. 
\begin{proposition}\label{newp}
Let $J\in \{\Z, \Z^+, \Z^-\}$. The following statements are equivalent:
\begin{enumerate}
\item \eqref{ldedd} admits an exponential dichotomy;
\item for $i\in \{1, 2\}$, the discrete system 
\begin{equation}\label{iid}
x_i(n+1)=A_{ii}(n)x_i(n) \quad n\in J'
\end{equation}
admits an exponential dichotomy.
\end{enumerate}
\end{proposition}

\begin{proof}

$(1)\implies (2)$:
We first consider the case when $J=\Z$. Suppose that~\eqref{lde-f} admits an exponential dichotomy. Take an arbitrary bounded sequence $(y_1(n))_{n\in \Z}\subset X_1$ and set $y(n)=(y_1(n), 0)\in X_1\times X_2$, $n\in \Z$. Then $(y(n))_{n\in \Z}$ is a bounded sequence. By Theorem~\ref{Hthm}, there exists a (unique) bounded sequence $(x(n))_{n\in \Z}\subset X_1\times X_2$ such that~\eqref{ADM-1} holds with
\[
A(n)=\begin{pmatrix}
A_{11}(n) & 0 \\
0 & A_{22}(n)
\end{pmatrix}.
\]
Writing $x(n)=(x_1(n), x_2(n))$ with $x_i(n)\in X_i$ for $i\in \{1, 2\}$, we have that 
\[
x_1(n+1)=A_{11}(n)x_1(n)+y_1(n), \quad n\in \Z.
\]
Moreover, the sequence $(x_1(n))_{n\in \Z}\subset X_1$ is bounded. Assume now that $(\tilde x_1(n))_{n\in \Z}\subset X_1$ is another bounded sequence with the property that 
\[
\tilde x_1(n+1)=A_{11}(n)\tilde x_1(n)+y_1(n), \quad n\in \Z.
\]
Let $\tilde x(n)=(\tilde x_1(n), 0)\in X_1\times X_2$, $n\in \Z$. Clearly $(\tilde x(n))_{n\in \Z}\subset X_1\times X_2$ is a bounded sequence and 
\[
\tilde x(n+1)=A(n)\tilde x(n)+y(n), \quad n\in \Z.
\]
This implies (using Theorem~\ref{Hthm}) that $x(n)=\tilde x(n)$, and thus $x_1(n)=\tilde x_1(n)$, for each $n\in \Z$. By Theorem~\ref{Hthm}, we conclude that~\eqref{iid} admits an exponential dichotomy for $i=1$. The same arguments also give the existence of an exponential dichotomy of~\eqref{iid} for $i=2$.

We now consider the case $J=\Z^+$. Proceeding as in the case $J=\Z$, since~\eqref{ldedd} admits an exponential dichotomy,  for each bounded sequence $(y_1(n))_{n\in J}\subset X_1$ there exists a bounded sequence $(x_1(n))_{n\in J}\subset X_1$ such that 
\[
x_1(n+1)=A_{11}(n)x_1(n)+y_1(n), \quad n\in J.
\]
Let $\Phi_{ii}(m, n)$ and $\Phi(m, n)$ denote the cocycles associated with~\eqref{iid} for $i=1, 2$ and~\eqref{ldedd}, respectively. Observe that $\mathcal S=\mathcal S_1\times \mathcal S_2$, where
\begin{equation}\label{mS}
\mathcal S=\left \{(v_1, v_2)\in X_1\times X_2: \ \sup_{n\in J}\|\Phi(n, 0)(v_1, v_2)\|<+\infty \right \}
\end{equation}
and 
\begin{equation}\label{Si}
\mathcal S_i=\left \{v_i\in X_i: \ \sup_{n\in J}\|\Phi_{ii}(n, 0)v_i\|<+\infty \right \}, \quad i=1, 2.
\end{equation}
Since~\eqref{lde-f} admits an exponential dichotomy, we find that $\mathcal S$ is closed, which easily implies that $\mathcal S_1$ is closed. Since $X_1$ is a Hilbert space, we see that $\mathcal S_1$ is complemented. 
Then Theorem~\ref{Hthm2} implies that~\eqref{iid} admits an exponential dichotomy for $i=1$. Analogously, \eqref{iid} admits an exponential dichotomy for $i=2$.

Finally, we treat the case $J=\mathbb Z^-$. Again, proceeding as in case $J=\mathbb Z$, we find that for each bounded sequence $(y_1(n))_{n\in J'}\subset X_1$, there exists a bounded sequence $(x_1(n))_{n\in J}\subset X_1$ satisfying
\[
x_1(n+1)=A_{11}(n)x_1(n)+y_1(n), \quad n\in J'.
\]
Moreover, $\mathcal U=U_1\times \mathcal U_2$, where $\mathcal U$ is the subspace of $X_1\times X_2$ consisting of all $(v_1, v_2)\in X_1\times X_2$ for which there is a bounded sequence $(x(n))_{n\in J}\subset X_1\times X_2$ with $x(0)=(v_1, v_2)$ and 
\begin{equation}\label{832}
x(n)=\begin{pmatrix}
A_{11}(n-1) & 0\\
0 & A_{22}(n-1)
\end{pmatrix}x(n-1) \quad n\in J,
\end{equation} and  $\mathcal U_i$ is the subspace of $X_i$ consisting of all $v_i \in X_i$ for which there is a bounded sequence $(x_i(n))_{n\in J}\subset X_i$ with $x_i(0)=v_i$ and $x_i(n)=A_{ii}(n-1)x_i(n-1)$ for $n\in J$ for $i=1, 2$. Since~\eqref{ldedd} admits an exponential dichotomy, we have that $\mathcal U$ is closed (see Theorem~\ref{Hthm3}). 
Hence, $\mathcal U_1$ is also closed and thus complemented (as $X_1$ is a Hilbert space). Finally, we take  a bounded sequence $(x_1(n))_{n\in J}\subset X_1$ such that $x_1(0)=0$ and $x_1(n)=A_{11}(n-1)x_1(n-1)$ for $n\in J$. Set $x(n):=(x_1(n), 0)\in X_1\times X_2$ for $n\in J$. Note that $(x(n))_{n\in J}$ is bounded, $x(0)=0$ and~\eqref{832} holds. Hence, $x(n)=0$ for $n\in J$ which implies that $x_1(n)=0$ for $n\in J$.
Using Theorem~\ref{Hthm3} we conclude that~\eqref{iid} admits an exponential dichotomy for $i=1$. Similarly, one can deal with the case $i=2$.

$(2)\implies (1)$: If~\eqref{iid} admits an exponential dichotomy on $J$ with projections $P_i(n)$, $n\in J$ for $i=1, 2$, it is straightforward to verify that~\eqref{lde-f} admits an exponential dichotomy with respect to projections 
\[
P(n)=\begin{pmatrix}
P_1(n) & 0\\
0 &P_2(n)
\end{pmatrix}, \quad n\in J.
\]

\end{proof}

\medskip

\section{From  dichotomy of block diagonal systems to dichotomy of block triangular systems}\label{Results1}
In this section, we explore the conditions under which the exponential dichotomy of the block-diagonal system~\eqref{ldedd} implies the existence of an exponential dichotomy for the block-triangular system~\eqref{LDE}. 

The proof of the following result is essentially identical to the proof of~\cite[Theorem 4]{DP} (in the case $J=\Z$). We include it for the sake of completeness. 
\begin{theorem}\label{T1}
Let $J=\mathbb Z$.
Suppose that~\eqref{ldedd} admits an exponential dichotomy and that 
\begin{equation}\label{bnd-1}
    \sup_{n\in J}\|A_{12}(n)\|<+\infty.
\end{equation}
Then, the system~\eqref{LDE} admits an exponential dichotomy.
\end{theorem}

\begin{proof}
Take an arbitrary bounded sequence $(y(n))_{n\in J}\subset X_1\times X_2$, and write $y(n)=(y_1(n), y_2(n))$ with $y_i(n)\in X_i$ for $n\in J$ and $i=1, 2$.
Clearly, $(y_1(n))_{n\in J}$ and $(y_2(n))_{n\in J}$ are also bounded sequences. Since the system~\eqref{iid} for $i=2$ admits an exponential dichotomy (see Proposition~\ref{newp}), it follows from Theorem~\ref{Hthm} that there exists a  (unique) bounded sequence $(x_2(n))_{n\in J}\subset X_2$ such that
\begin{equation}\label{po}
x_2(n+1)=A_{22}(n)x_2(n)+y_2(n) \quad n\in J.
\end{equation}
Moreover, it follows from~\eqref{bnd-1}  that $(A_{12}(n)x_2(n))_{n\in J}$ is a bounded sequence. Hence, by applying Theorem~\ref{Hthm} to system~\eqref{iid} with $i=1$, we conclude that there exists a 
(unique) bounded sequence $(x_1(n))_{n\in J}$ in $X_1$ satisfying
\begin{equation}\label{po2}
x_1(n+1)=A_{11}(n)x_1(n)+A_{12}(n)x_2(n)+y_1(n) \quad n\in J.
\end{equation}
Then
\[
x(n):=(x_1(n), x_2(n))\in X_1\times X_2, \quad n\in J,
\]
is a bounded solution of \eqref{ADM-1}, where $A(n)$ is as in~\eqref{tf}.

It remains to establish the uniqueness. Suppose that $(\tilde x(n))_{n\in J}$ is another bounded solution of~\eqref{ADM-1}. Then, writing $\tilde{x}(n)=(\tilde{x}_1(n), \tilde{x}_2(n))$ with $\tilde{x}_i(n)\in X_i$ for $i=1, 2$, we have
\[
\begin{split}
(x_1-\tilde{x}_1)(n+1)&=A_{11}(n)(x_1-\tilde{x}_1)(n)+A_{12}(n)(x_2-\tilde{x}_2)(n) \\ (x_2-\tilde{x}_2)(n+1)&=A_{22}(n)(x_2-\tilde{x}_2)(n),
\end{split}
\]
for $n\in J$.
Since~\eqref{iid} admits an exponential dichotomy for $i=2$ (see Proposition~\ref{newp}), from the second equality above we conclude that $x_2(n)=\tilde{x}_2(n)$ for $n\in J$. Then, from the first equality above and the exponential dichotomy of~\eqref{iid} for $i=1$ it follows that $x_1(n)=\tilde{x}_1(n)$ for each $n\in J$. We conclude that $x(n)=\tilde{x}(n)$ for $n\in J$.
Applying Theorem~\ref{Hthm}, we conclude that~\eqref{lde-f} admits an exponential dichotomy.
\end{proof}

We now consider dichotomies on $\Z^+$ and $\Z^-$. 
\begin{theorem}\label{129thm}
Let $J \in \{\Z^+,\Z^-\}$. If~\eqref{ldedd} admits an exponential dichotomy and~\eqref{bnd-1} holds, 
then~\eqref{LDE} admits an exponential dichotomy.
\end{theorem}

\begin{proof}
Let $J=\Z^+.$ The proof for the case $J=\Z^-$ is similar. As in the proof of the previous theorem, we conclude that for each bounded sequence $(y(n))_{n\in J}\subset X_1\times X_2$, there exists a bounded sequence $(x(n))_{n\in J}\subset X_1\times X_2$ in $X_1 \times X_2$ satisfying~\eqref{ADM2-2} with $A(n)$ as in~\eqref{tf}. We now claim that 
\begin{equation}\label{S}
\mathcal S:=\left \{(x_1, x_2)\in X_1\times X_2: \ \sup_{n\in J} \|\Phi(n, 0)(x_1, x_2) \|<+\infty \right \}
\end{equation}
is a closed (and complemented) subspace of $X_1\times X_2$, where $\Phi(m, n)$ denotes the cocycle associated with~\eqref{LDE}.
Observe that $((x_1(n), x_2(n)))_{n\in J}\in X_1\times X_2$ is a bounded solution of~\eqref{LDE} if and only if $x_2(n)=\Phi_{22}(n, 0)\eta$ and 
\[
\begin{split}
x_1(n) &=\Phi_{11}(n, 0)\xi+\sum_{k=1}^n \Phi_{11}(n, k)P_{11}(k)A_{12}(k-1)x_2(k-1) \\
&\phantom{=}-\sum_{k=n+1}^\infty \Phi_{11}(n, k)(\Id_1-P_{11}(k))A_{12}(k-1)x_2(k-1),
\end{split}
\]
for $n\in J$ with $\eta \in \mathcal RP_{22}(0)$ and $\xi \in \mathcal RP_{11}(0)$, where $P_{ii}(n)$, $n\in J$ are projections associated with the exponential dichotomy of~\eqref{iid}, $\Phi_{ii}(m, n)$ is the cocycle associated with~\eqref{iid} for $i=1, 2$ and $\Id_1$ denotes the identity on $X_1$. This follows from the observation that the sequence $(x_1(n))_{n\in J}$ is the unique solution of 
\[
x_1(n+1)=A_{11}(n)x_1(n)+A_{12}(n)x_2(n), \quad n\in J
\]
satisfying $P_{11}(0)x_1(0)=\xi$. Taking the above into account, we conclude that \[
\mathcal S=\{(\xi+L\eta, \eta): \ \xi\in \mathcal RP_{11}(0), \ \eta\in \mathcal RP_{22}(0) \},
\]
where $L\colon \mathcal RP_{22}(0)\to \mathcal NP_{11}(0)$ is given by 
\[
L\eta:=-\sum_{k=1}^\infty \Phi_{11}(0, k)(\Id_1-P_{11}(k))A_{12}(k-1)\Phi_{22}(k-1, 0)\eta, \quad \eta \in \mathcal RP_{22}(0).
\]
It follows easily from the dichotomy estimates related to~\eqref{iid} for $i=1, 2$ and~\eqref{bnd-1} that $L$ is a bounded linear operator. We consider $P\in \mathcal L(X_1\times X_2)$ given by 
\[
P=\begin{pmatrix}
P_{11}(0) & LP_{22}(0) \\
0& P_{22}(0)
\end{pmatrix}.
\]
As $P_{11}(0)L=0$, one can easily verify that $P$ is a projection. Moreover, we clearly have $\mathcal S=\mathcal RP$, which establishes the desired claim. \end{proof}

\begin{remark}\label{Remark5}
Compared with the proof of~\cite[Theorem 4]{DP} (in the case $J=\Z^+$), in the proof of Theorem~\ref{129thm} it was necessary to show that $\mathcal S$ in~\eqref{S} is closed (and complemented) so that we can apply Theorem~\ref{Hthm2}.

\end{remark}

\section{From exponential dichotomy of block triangular systems to exponential dichotomy of block diagonal systems}\label{Results2}

We are now interested in formulating conditions under which the existence of an exponential dichotomy for~\eqref{LDE} implies the existence of an exponential dichotomy of~\eqref{ldedd}. The following examples will illustrate that the situation is much more delicate than in the finite-dimensional case.
\begin{example}
Let $\ell_2$ denote the space of all one-sided sequences $(z_n)_{n\in \Z^+}\subset \mathbb C$ such that $\sum_{n=0}^\infty |z_n|^2<+\infty$, and recall that  $\ell_2$ is a Hilbert space  with respect to the scalar product
\[
\langle (z_0, z_1, \ldots), (w_0, w_1, \ldots)\rangle=\sum_{n=0}^\infty z_n \overline{w}_n.
\]
Let $U\colon \ell_2 \to \ell_2$ be the unilateral shift operator defined by
\[
U(z_0, z_1, \ldots, )=(0, z_0, z_1, \ldots), \quad (z_0, z_1, \ldots)\in \ell_2.
\] Finally, we consider $C\colon \ell_2\to \ell_2$  given by
\[
C(x_0, x_1, \ldots, )=(x_0, 0, 0, \ldots), \quad (x_0, x_1, \ldots)\in \ell_2.
\]
Observe that $U$ and $C$ are bounded linear operators. 
Then, we have (see~\cite[Example 3]{KJ}) \[
\sigma \left (\begin{pmatrix}
U & C\\
0 & U^*
\end{pmatrix}
\right )\subset \{z\in \mathbb C: \ |z|=1\},
\]
and
\[
\sigma(U)=\sigma(U^*)=\{z\in \mathbb C: \ |z| \le 1\}.
\]
We take an arbitrary $\lambda >1$ and set
\[
A_{11}=A_{11}(n)=\lambda U, \quad A_{22}=A_{22}(n)=\lambda U^* \quad \text{and} \quad  A_{12}=A_{12}(n)=\lambda C,
\]
for $n\in J\in \{\Z, \Z^+, \Z^-\}$.
 Then, \eqref{LDE} admits an exponential dichotomy, since the spectrum of $A=A(n)$ as in~\eqref{tf} is disjoint from the unit circle in $\mathbb C$, 
 while~\eqref{iid} does not admit an exponential dichotomy for $i=1, 2$ as the spectra of $A_{ii}$ intersect the unit circle. 

 This example shows that the existence of an exponential dichotomy of~\eqref{LDE} in general does not imply the existence of an exponential dichotomy of~\eqref{ldedd} even in the autonomous case. 
\end{example}

Next, we provide a similar example in which the operators $A_{ii}=A_{ii}(n)$ are invertible for $i=1, 2$. This example was communicated to us by Nikolaos Chalmoukis using MathOverflow.

 \begin{example}\label{EX2}
We consider the space $X=L^2(\partial A_r, |dz|)$ of square-integrable functions with respect to the arc length, where $A_r=\{z \in \mathbb C : r^{-1} < |z| < r \}$ is the annulus for some $r>1$. Let $H^2(\partial A_r)$ be the Hardy space of the annulus, i.e. the closure of the span of the set of functions $\{z \mapsto z^n : n \in \mathbb Z\}$. We consider $A \colon X \to X$ to be the multiplication operator by $z$, that is,  $$A f(z)=z f(z), \text{ for every } f \in X.$$ It is well known that $\sigma(A)=\partial A_r$. Since $H^2(\partial A_r)$ is $A$-invariant, we can decompose $A$ into the following block upper triangular form with respect to the splitting $X=H^2(\partial A_r)\oplus H^2(\partial A_r)^\perp$:
$$A= \begin{pmatrix}
    A_{11} & A_{12} \\ 0 & A_{22}
\end{pmatrix}.$$
It is not difficult to show that $\sigma(A_{11})=\sigma (A_{22})=\overline{A_r}$ (see~\cite{NC} for details). 

We now consider~\eqref{LDE} with $A(n)=A$ for $n\in J\in \{\Z, \Z^-, \Z^+\}$. Since $\sigma(A_{11})=\sigma(A_{22})=\overline{A_r}$ we find that \eqref{iid} does not admit an exponential dichotomy for $i=1,2$ (as $\overline{A_r}$ intersects the unit circle in the complex plane), while \eqref{LDE} admits an exponential dichotomy (since $\sigma(A)=\partial A_r$ does not intersect the unit circle in the complex plane). 

 \end{example}

Despite the above examples, we will formulate a series of results that give additional conditions under which the existence of an exponential dichotomy for~\eqref{LDE} implies the existence of an exponential dichotomy for~\eqref{ldedd}.

\subsection{Dichotomies on $\Z$}
We begin with the following result whose proof is identical to the proof of~\cite[Theorem 6]{DP}. It is included for the sake of completeness. 
\begin{theorem}\label{tconv1}
Let $J=\Z$ and suppose that~\eqref{LDE} admits an exponential dichotomy. Then, \eqref{ldedd} admits an exponential dichotomy if and only if~\eqref{iid} admits no nonzero bounded solutions for $i=2$. 
\end{theorem}

\begin{proof}
The necessity is obvious.
For sufficiency, 
using Proposition~\ref{newp}, it suffices to show that~\eqref{iid} admits an exponential dichotomy for $i=1,2$. We first deal with the case $i=2$.
Let $(y_2(n))_{n\in J}\subset X_2$ be an arbitrary bounded sequence in $X_2$. Set $y(n):=(0, y_2(n))\in X_1\times X_2$ for $n\in J$. Then $(y(n))_{n\in J}\subset X_1\times X_2$ is a bounded sequence.  Since~\eqref{LDE} admits an exponential dichotomy, by Theorem~\ref{Hthm} there exists a (unique) bounded sequence $(x(n))_{n\in J}\subset X_1\times X_2$ such that~\eqref{ADM-1} holds with $A(n)$ as in~\eqref{tf}. In particular, by
writing $x(n)=(x_1(n), x_2(n))$, where $x_i(n)\in X_i$ for $i=1, 2$,  we have 
\[
x_2(n+1)=A_{22}(n)x_2(n)+y_2(n), \quad n\in J.
\]
Moreover,  $(x_2(n))_{n\in J}$ is bounded. Hence, since~\eqref{iid} for $i=2$ does not admit nonzero bounded solutions, we conclude (using Theorem~\ref{Hthm}) that it does admit an exponential dichotomy.

We now turn to~\eqref{iid} for $i=1$.
Take an arbitrary bounded sequence $(y_1(n))_{n\in J}\subset X_1$. Set $y(n):=(y_1(n), 0)\in X_1\times X_2$ for $n\in J$. Then $(y(n))_{n\in J}$ is a bounded sequence. Since~\eqref{LDE} admits an exponential dichotomy, there exists a (unique) bounded sequence $(x(n))_{n\in J}\subset X_1\times X_2$ satisfying~\eqref{ADM-1}.
Writing $x(n)=(x_1(n), x_2(n))$, where $x_i(n)\in X_i$ for $i=1, 2$, we have that 
\begin{equation}\label{ab1-1}
x_1(n+1)=A_{11}(n)x_1(n)+A_{12}(n)x_2(n)+y_1(n) 
\end{equation}
and
\begin{equation}\label{ab2-2}
x_2(n+1)=A_{22}(n)x_2(n),
    \end{equation}
for $n\in J$. Since $(x_2(n))_{n\in \Z}$ is bounded, \eqref{ab2-2} together with our assumption that~\eqref{iid} for $i=2$ does not admit nonzero bounded solutions implies that $x_2(n)=0$ for $n\in J$. Hence, \eqref{ab1-1} reduces to
\begin{equation}\label{y-1}
x_1(n+1)=A_{11}(n)x_1(n)+y_1(n) \quad n\in J,
\end{equation}
where the sequence $(x_1(n))_{n\in J}$ is bounded. Suppose that there exists another bounded sequence $(\tilde x_1(n))_{n\in J}\subset X_1$ such that 
\[
\tilde x_1(n+1)=A_{11}(n) \tilde x_1(n)+y_1(n) \quad n\in J.
\]
Then,
\[
(\tilde x_1(n+1), 0)=A(n)(\tilde x_1(n), 0)+y(n), \quad n\in J.
\]
Furthermore, \eqref{y-1} implies that 
\[
(x_1(n+1), 0)=A(n)( x_1(n), 0)+y(n), \quad n\in J.
\]
Since~\eqref{LDE} admits an exponential dichotomy, we conclude that $x_1(n)=\tilde x_1(n)$ for all $n\in J$. By Theorem~\ref{Hthm}, we conclude that~\eqref{iid} for $i=1$ admits an exponential dichotomy. The proof of the theorem is complete.
\end{proof}

\subsection{Dichotomies on $\Z^+$ and $\Z^-$}
Before formulating versions of Theorem~\ref{tconv1} for $J\in \{\Z^+, \Z^-\}$, we will establish two auxiliary lemmas.

\begin{lemma}\label{diciid2}
Let $J=\Z^+$, and assume that~\eqref{LDE} admits an exponential dichotomy. If $\mathcal S_2$ given by~\eqref{Si} for $i=2$ is complemented in $X_2$, then~\eqref{iid} admits an exponential dichotomy for $i=2$.
\end{lemma}

\begin{proof} Arguing as in the proof Theorem~\ref{tconv1}, we have that for each bounded sequence $(y_2(n))_{n\in J}\subset X_2$, there exists a bounded sequence $(x_2(n))_{n\in J}$ such that 
\[
x_2(n+1)=A_{22}(n)x_2(n)+y_2(n), \quad n\in J.
\]
Since $\mathcal S_2$ is complemented in $X_2$, it follows from Theorem~\ref{Hthm2} and Remark~\ref{120rem} that~\eqref{iid} admits an exponential dichotomy for $i=2$.
\end{proof}

\begin{lemma}\label{neglin}
Let $J=\Z^-$ and assume that~\eqref{LDE} admits an exponential dichotomy. Furthermore, suppose that the following holds:
\begin{enumerate}
\item the subspace $\mathcal U_2$ of $X_2$ consisting of all $v\in X_2$ for which there is a bounded sequence $(x_2(n))_{n\in J}$ with $x_2(0)=v$ and $x_2(n)=A_{22}(n-1)x_2(n-1)$ for $n\le 0$ is complemented in $X_2$;
\item there are no nonzero bounded solutions $(x_2(n))_{n\in J}$ of~\eqref{iid} for $i=2$ such that $x_2(0)=0$.
\end{enumerate}
Then \eqref{iid} admits an exponential dichotomy for $i=2$.

\end{lemma}

\begin{proof}
The proof is analogous to the proof of the previous lemma, using Theorem~\ref{Hthm3} (and Remark~\ref{1037rem}) instead of Theorem~\ref{Hthm2}. 
\end{proof}

\begin{theorem}\label{thm7}
Let $J=\mathbb Z^+$ and suppose that~\eqref{LDE} admits an exponential dichotomy with projections $P(n)$ such that $d:=\dim \mathcal N P(0)<+\infty$. Furthermore, let $\mathcal S_i$, $i=1, 2$ be given by~\eqref{Si}.  Then, the following holds:
\begin{enumerate}
    \item  $\mathcal S_2$ is complemented and closed,  $d_2:=\codim \mathcal S_2 \le d$ and~\eqref{iid} admits an exponential dichotomy for $i=2$;
    \item $\mathcal S_1$ given by~\eqref{Si} for $i=1$ is closed and complemented;
    \item \eqref{ldedd} admits an exponential dichotomy if and only if
    \begin{equation}\label{dddeq}d=d_1+d_2,\end{equation} where $d_1:=\codim \mathcal S_1$.
    \end{enumerate} 
\end{theorem}

\begin{proof}
$(1)$ Take an arbitrary $x_2\in X_2$. Then $(0, x_2)\in X_1\times X_2$, and consequently
\[
(0, x_2)=(y_1, y_2)+(z_1, z_2),
\]
for $(y_1, y_2)\in \mathcal R P(0)$ and $(z_1, z_2)\in \mathcal NP(0)=:Z$. Note that $\mathcal RP(0)=\mathcal S$, where $\mathcal S$ is given by~\eqref{mS} .
Hence, $y_2\in \mathcal S_2$ and $x_2\in \mathcal S_2+\pi_2 (Z)$, where $\pi_2\colon X_1\times X_2 \to X_2$ denotes the projection on the second coordinate. We conclude that $X_2=\mathcal S_2+\pi_2(Z)$. Let $W\subset \pi_2(Z)$ be a subspace such that \[X_2=\mathcal S_2\oplus W.\] Note that $\pi_2(Z)$ is finite-dimensional (as $Z$ is), and thus $W$ is finite-dimensional and consequently closed. Hence, $\mathcal S_2$ is complemented. By Lemma~\ref{diciid2}, we conclude that~\eqref{iid} admits an exponential dichotomy for $i=2$. This immediately implies that $\mathcal S_2$ is also closed. Finally, observe that 
\[d_2=\codim \mathcal S_2=\dim W\le \dim \pi_2(Z) \le \dim Z=d.\]

$(2)$ Let $(x_n)_{n\in \mathbb N}$ be a sequence in $\mathcal S_1$ that converges to $x\in X_1$. Note that $((x_n, 0))_{n\in \mathbb N}\subset X_1\times X_2$ is a sequence in $\mathcal S$ that converges to $(x, 0)\in X_1\times X_2$. Since $\mathcal S=\mathcal RP(0)$ is closed, we have $(x, 0)\in \mathcal S$, which gives $x\in \mathcal S_1$. This shows that $\mathcal S_1$ is closed. Since $X_1$ is a Hilbert space, we see that it is complemented.

$(3)$ First we suppose that $d=d_1+d_2$ with $d_1=\codim \mathcal S_1$. Choose subspaces $Z_i\subset X_i$ such that $X_i=\mathcal S_i\oplus Z_i$ and $\dim Z_i=d_i$ for $i=1, 2$. We claim that
\[
X_1\times X_2=\mathcal S \oplus (Z_1\times Z_2).
\]
It is enough to prove that $\mathcal S \cap (Z_1 \times Z_2) = \{0\}$ since \[\codim \mathcal S=\dim \mathcal NP(0)=d=d_1+d_2=\dim (Z_1\times Z_2).\]
Take $(v_1,v_2) \in \mathcal S \cap (Z_1 \times Z_2)$. Then $\sup_{n \in \Z^+} \|\Phi(n,0)(v_1,v_2)\| < + \infty$, and therefore $\sup_{n \in \Z^+} \|\Phi_{22}(n,0)v_2\| < + \infty$, also. From here we get that $v_2 \in S_2 \cap Z_2$, and therefore $v_2=0$. Since $\sup_{n \in \Z^+} \|\Phi(n,0)(v_1, v_2)\| < + \infty$, we get $\sup_{n \in \Z^+} \|\Phi_{11}(n,0)v_1 + \Phi_{12}(n,0)v_2\|=\sup_{n \in \Z^+} \|\Phi_{11}(n,0)v_1\| < + \infty$. Hence, $v_1 \in S_1 \cap Z_1$, that is, $v_1=0$.

    Let now $ (y_1(n))_{n\in J}\subset X_1$ be an arbitrary bounded sequence. We define $y(n)=(y_1(n),0)\in X_1\times X_2$, $n \in J$. Then $(y(n))_{n\in J}$ is bounded. By Theorem~\ref{Hthm2} we see that there is a bounded sequence $(x(n))_{n\in J}\subset X_1\times X_2$ that satisfies~\eqref{ADM2-2} with $A(n)$ as in~\eqref{tf}. Without loss of generality, we may assume that $x(0)\in Z_1\times Z_2$. In fact, otherwise, we write $x(0)$ as the sum of a vector from $\mathcal S$ that we denote by $(v_1, v_2)$ and a vector from $Z_1\times Z_2$, and replace $(x(n))_{n\in J}$ by the sequence $(\bar x(n))_{n\in J}$ given by $\bar x(n):=x(n)-\Phi(n, 0)(v_1, v_2)$. It follows from~\eqref{ADM2-2} that (writing $x(n)=(x_1(n), x_2(n))$ with $x_i(n)\in X_i$, $i=1, 2$),
   \begin{equation}\label{fw1}x_1(n+1)=A_{11}(n) x_1(n) + A_{12}(n) x_2(n) + y_1(n) \end{equation}
   and \begin{equation}\label{fw2} x_2(n+1)=A_{22}(n) x_2(n),\end{equation} for $n \in J$. 
   Since $(x_2(n))_{n\in J}$ is bounded, from~\eqref{fw2} we obtain $x_2(0)\in \mathcal S_2$. On the other hand, $x_2(0)\in Z_2$ as $x(0)\in Z_1\times Z_2$. Therefore, $x_2(0)\in \mathcal S_2\cap Z_2$, that is, $x_2(0)=0$.  We conclude that $x_2(n)=0$ for $n \in J$. Thus, \eqref{fw1} reduces to \begin{equation}\label{reduction}x_1(n+1)=A_{11}(n)x_1(n)+y_1(n), \end{equation}
   for $n\in J$.  Note that the sequence $(x_1(n))_{n\in J}$ is bounded.
   Since $\mathcal S_1$ is closed and complemented, we have (using Theorem~\ref{Hthm2}) that ~\eqref{iid} admits an exponential dichotomy for $i=1$. Hence, \eqref{ldedd} admits an exponential dichotomy.

   We now establish the converse statement. We suppose that~\eqref{ldedd} admits an exponential dichotomy. Let $P_{ii}(n)$, $n\in J$ denote projections associated with the exponential dichotomy of~\eqref{iid} for $i=1, 2$.
\begin{lemma}\label{ddd}
We have $d\ge d_1+d_2$.
\end{lemma}

\begin{proof}[Proof of the lemma]
Let $\mathcal S_2'$ be the subspace of $\mathcal S_2=\mathcal RP_{22}(0)$ consisting of all $\eta \in \mathcal S_2$ for which the equation
\begin{equation}\label{918eq}
x_1(n+1)=A_{11}(n)x_1(n)+A_{12}(n)\Phi_{22}(n, 0)\eta \quad n\in J,
\end{equation}
has a bounded solution $(x_1(n))_{n\in J}\subset X_1$.
For such $\eta$, the above equation has a unique bounded solution $(x_1(n))_{n\in J}\subset X_1$ with $x_1(0)\in \mathcal N P_{11}(0)$. Set $L\eta:=x_1(0)$. Then $L\colon \mathcal S_2\to \mathcal NP_{11}(0)$ is a linear operator, and
\[
\mathcal R P(0)=\{(\xi+L\eta, \eta): \ \xi\in \mathcal RP_{11}(0), \ \eta \in \mathcal S_2'\}.
\]
We take an arbitrary (not necessarily closed) subspace $W_2$ of $X_2$ such that $X_2=\mathcal S_2'\oplus W_2$. Then
\[ X_1\times X_2=\mathcal RP(0)\oplus (\mathcal NP_{11}(0)\times W_2), \] which implies
\begin{equation}\label{som}
d=d_1+\dim W_2=d_1+\codim \mathcal S_2'\ge d_1+d_2,
\end{equation}
as $\mathcal S_2'\subset \mathcal S_2$.
\end{proof}

In order to establish the reverse inequality, we consider the discrete system:
\begin{equation}\label{LDEb}
x(n+1)=B(n)x(n) \quad n\in \Z^-\setminus \{0\},
\end{equation}
where \[B(n)=A(-n-1)^*, \quad  n\in \Z^-\setminus \{0\}.\] Observe that 
\[
B(n)=\begin{pmatrix}
B_{11}(n) & 0 \\
B_{21}(n) & B_{22}(n)
\end{pmatrix},
\]
where 
\[
B_{11}(n)=A_{11}(-n-1)^*, \  B_{21}(n)=A_{12}(-n-1)^* \ \text{and} \ B_{22}(n)=A_{22}(-n-1)^*,
\]
for $n\in \Z^-\setminus \{0\}$. 
\begin{lemma}\label{lemadj}
The system~\eqref{LDEb} admits an exponential dichotomy with respect to projections
\[
\tilde P(n):=P(-n)^*, \quad n\in \Z^-.
\]
\end{lemma}
\begin{proof}[Proof of the lemma]
Set
\[
\bar \Phi (m, n):=\Phi(n, m)^* \quad \text{for $m, n\in \Z^+$ with $m\le n$.}
\]
Then, the following can be  verified by using the same arguments as in~\cite[Proposition 2.2.]{RZ} (see also~\cite[p.229]{Lin}):
\begin{itemize}
\item $\bar \Phi(m, n)P(n)^*=P(m)^*\bar \Phi(m, n)$ for $0\le m\le n$;
\item $\bar{\Phi}(m, n)\rvert_{\mathcal N P(n)^*}\colon \mathcal N P(n)^* \to \mathcal N P(m)^*$ is invertible for $0\le m\le n$;
\item for $0\le m\le n$,
\begin{equation}\label{dica1}
\| \bar \Phi(m, n)P(n)^*\| \le De^{-\lambda (n-m)};
\end{equation}
\item for $m\ge n\ge 0$,
\[
\| \bar \Phi(m, n)(\Id-P(n)^*)\| \le De^{-\lambda (m-n)},
\]
where 
\[
\bar \Phi(m, n):=\left (\bar \Phi(n, m)\rvert_{\mathcal N P(m)^*}\right )^{-1} \quad m\ge n\ge 0,
\]
and $D, \lambda>0$ are constants associated with the exponential dichotomy of~\eqref{LDE}.
\end{itemize}
Let $\Phi_B$ denote the cocycle associated to~\eqref{LDEb}. Observe that
\[
\Phi_B(m, n)=\Phi(-n, -m)^*=\bar \Phi(-m, -n), \quad 0\ge m\ge n.
\]
Consequently, for $0\ge m \ge n$ we have
\[
\begin{split}
\tilde P(m)\Phi_B(m, n)=P(-m)^*\bar \Phi(-m, -n) &=\bar \Phi(-m, -n)P(-n)^* \\
&=\Phi_B(m, n)\tilde P(n).
\end{split}
\]
Furthermore, $\Phi_B(m, n)\rvert_{\Ker \tilde P(n)}=\bar \Phi(-m, -n)\rvert_{\Ker P(-n)^*}$ is invertible for $0\ge m\ge n$.
Next, note that~\eqref{dica1} implies that
\[
\| \Phi_B(m, n)\tilde P(n)\|=\|\bar \Phi(-m, -n)P(-n)^*\| \le De^{-\lambda (-n-(-m))}=De^{-\lambda (m-n)},
\]
for $0\ge m\ge n$. Similarly, 
\[
\|\Phi_B(m, n)(\Id-\tilde P(n))\| \le De^{-\lambda (n-m)}, \quad \text{for $m\le n\le 0$.}
\]
\end{proof}
In the same manner, for $i=1, 2$, the discrete system 
\begin{equation}\label{iid-b}
x_{ii}(n+1)=B_{ii}(n)x_{ii}(n), \quad n\in \Z^-\setminus \{0\}
\end{equation}
admits an exponential dichotomy with respect to projections $\tilde P_{ii}(n)$, where 
\[
\tilde P_{ii}(n)=P_{ii}(-n)^*, \quad n\in \Z^-, \ i=1, 2
.\]
\begin{lemma}\label{832lem}
We have
\[
\dim \mathcal N \tilde P(0)=\dim \mathcal N P(0).
\]
\end{lemma}
\begin{proof}[Proof of the lemma]
Note that 
\[
(x_1, x_2)\in \mathcal N \tilde P(0) \iff (x_1, x_2)\in \mathcal N P(0)^* 
\iff (x_1, x_2)\in (\mathcal R P(0))^\perp,
\]
as $\mathcal N P(0)^*=(\mathcal RP(0))^\perp$. 
Thus,
\[
\dim \mathcal N  \tilde P(0)=\dim (\mathcal R P(0))^\perp=\dim \mathcal N P(0),
\]
since \[X_1\times X_2=\mathcal R P(0)\oplus \mathcal N P(0)=\mathcal R P(0)\oplus (\mathcal R P(0))^\perp.\]
\end{proof}
In the same manner,
\[
\dim \mathcal N \tilde P_{ii}(0)=\dim \mathcal N P_{ii}(0), \quad i=1, 2.
\]
Let $\mathcal U_1'$ denote the subspace which consists of all $\xi \in \mathcal N \tilde P_{11}(0)$ for which there is a bounded solution $(x_2(n))_{n\in \Z^-}\subset X_2$ of the equation
\[
x_2(n+1)=B_{22}(n)x_2(n)+B_{21}(n)x_1^\xi(n) \quad n\in \Z^-\setminus \{0\},
\]
where $(x_1^\xi(n))_{n\in \Z^-}$ is the unique bounded sequence in $X_1$ satisfying $x_1^\xi(n+1)=B_{11}(n)x_1^\xi(n)$ for $n\le -1$ and $x_1^\xi(0)=\xi$. Note that the above equation has a unique solution $(x_2(n))_{n\in \Z^-}$ with $x_2(0)\in \mathcal R \tilde P_{22}(0)$. We define a linear operator $L\colon \mathcal U_1'\to \mathcal R \tilde P_{22}(0)$ by $L\xi=x_2(0)$.  Then, it is easy to verify that 
\[
\mathcal N \tilde P(0)=\{(\xi, L\xi+\eta): \ \xi \in \mathcal U_1', \  \eta \in \mathcal N \tilde P_{22}(0)\}.
\]
Thus,\[
\dim \mathcal N \tilde P(0)=\dim \mathcal U_1'+\dim \mathcal N \tilde P_{22}(0)\le \dim \mathcal N \tilde P_{11}(0)+\dim \mathcal N \tilde P_{22}(0),
\] and hence
\[
d=\dim \mathcal N P(0)\le \dim \mathcal N P_{11}(0)+\dim \mathcal N  P_{22}(0)=d_1+d_2.
\]
This together with Lemma~\ref{ddd} yields~\eqref{dddeq}.
\end{proof}

\begin{remark}\label{1157rem}
\begin{enumerate}
\item We note that condition~\eqref{dddeq} can be reformulated as follows: for each $\eta \in \mathcal S_2$, the equation~\eqref{918eq} has a bounded solution $(x_1(n))_{n\in \Z^+}\subset X_1$. 
In fact, let $\mathcal S_2'$ be as in the proof of Lemma~\ref{ddd}. Then for each $\eta \in \mathcal S_2$, equation~\eqref{918eq} has a bounded solution $(x_1(n))_{n\in \Z^+}\subset X_1$ if and only if $\mathcal S_2'=\mathcal S_2$, which is equivalent to~\eqref{dddeq} by~\eqref{som}.
\item We observe that the requirement $\dim \mathcal NP(0)<+\infty$ is satisfied whenever $A(n)$ is a compact operator for some $n\in \Z^+$. 
\end{enumerate}
\end{remark}

\begin{remark}\label{Remark7}
Theorem~\ref{thm7} can be seen as a generalization of~\cite[Theorem 7]{DP} which deals with the case where $X_i$ is finite-dimensional for $i=1, 2$. In that setting~\eqref{dddeq} is equivalent to the condition:
\[
\dim \mathcal S=\dim \mathcal S_1+\dim \mathcal S_2,
\]
which coincides with~\cite[Eq. (2.16)]{DP}.

We emphasize that the proof of Theorem~\ref{thm7} is much more involved than the proof of~\cite[Theorem 7]{DP}, in particular the part showing that the existence of exponential dichotomies of both~\eqref{LDE} and~\eqref{ldedd} implies~\eqref{dddeq}. For this, we cannot rely on the arguments from~\cite[Remark 4.2]{BFP2} as in the infinite-dimensional setting it is not possible to approximate system~\eqref{LDE} with a system with invertible coefficients. 
\end{remark}

Let $J=\mathbb Z^-$. For $m\in J$ and $i\in \{1, 2\}$, let $U_i(m)$ denote the subspace of $X_i$ consisting of all $v_i\in X_i$ for which there is a sequence $(x_i(n))_{m\le n}\subset X_i$ satisfying $x_i(m)=v_i$ and $x_i(n)=A_{ii}(n-1)x_i(n-1)$ for $n\le m$. The following result is a generalization of~\cite[Theorem 8]{DP}.
\begin{theorem}\label{517thm}
Let $J=\mathbb Z^-$ and suppose that~\eqref{LDE} admits an exponential dichotomy with projections $P(n)$ such that $d:=\dim \mathcal N P(0)<+\infty$. Then, the following holds:
\begin{enumerate}
    \item  \eqref{iid}  admits an exponential dichotomy for $i=1$,  \[d_1(m):=\dim \mathcal U_1(m)\le d,\]  and $d_1(m)$ is independent of $m$;
    \item \eqref{ldedd} admits an exponential dichotomy if and only if
    \begin{equation}\label{dddeq2}d=d_1(m)+d_2(m) \quad \text{for $m\in J$},\end{equation} where $d_2(m):=\dim \mathcal U_2(m)$.
    \end{enumerate} 

\end{theorem}

\begin{proof}
$(1)$ We consider the discrete system:
\begin{equation}\label{eq851}
    x(n+1)=B(n)x(n) \quad n\in \Z^+,
\end{equation}
where 
\begin{equation}\label{555eq}
B(n)=A(-n-1)^*, \quad n\in \Z^+.
\end{equation}
Then 
\[
B(n)=\begin{pmatrix}
B_{11}(n) & 0\\
B_{21}(n) & B_{22}(n)
\end{pmatrix},
\]
where 
\[
B_{11}(n)=A_{11}(-n-1)^*, \ B_{21}(n)=A_{12}(-n-1)^* \ \text{and} \ B_{22}(n)=A_{22}(-n-1)^*,
\]
for $n\in \Z^+$. Using the arguments as in the proof of Lemma~\ref{lemadj}, we find that~\eqref{eq851} admits an exponential dichotomy with respect to projections $\tilde P(n)$, where $\tilde P(n)=P(-n)^*$, $n\in \Z^+$. As in Lemma~\ref{832lem}, $\dim \mathcal N \tilde P(0)=\dim \mathcal N P(0)=d$.

Using Theorem~\ref{Hthm2} and the exponential dichotomy of~\eqref{eq851}, it can be easily shown that for each bounded sequence $(y_1(n))_{n\in \Z^+}\subset X_1$, there exists a bounded sequence $(x_1(n))_{n\in \Z^+}\subset X_1$ such that
\[
x_1(n+1)=B_{11}(n)x_1(n)+y_1(n), \quad n\in \Z^+.
\]
Let 
\begin{equation}\label{602s1}
\mathcal S_1^B=\left \{v_1\in X_1: \sup_{n\in \Z^+}\|\Phi_{11}^B(n, 0)v_1\|<+\infty \right \},
\end{equation}
where $\Phi_{11}^B(m, n)$ denotes the cocycle associated with the discrete system 
\begin{equation}\label{B11}
x_{11}(n+1)=B_{11}(n)x_{11}(n), \quad n\in \Z^+.
\end{equation}
Take $x_1\in X_1$ and write $(x_1, 0)$ as
\[
(x_1, 0)=(y_1, y_2)+(z_1, z_2),
\]
where $(y_1, y_2)\in \mathcal R\tilde P(0)$ and $(z_1, z_2)\in \mathcal N \tilde P(0)$. Note that
$x_1=y_1+z_1$, 
$y_1\in \mathcal S_1^B$ and $z_1 \in \pi_1 \mathcal N \tilde P(0)=:Z$, where $\pi_1\colon X_1\times X_2\to X_1$ is the projection on the first coordinate. Hence, 
\[
X_1=\mathcal S_1^B+Z,
\]
and $\dim Z\le d$. Consequently, we can choose $W\subset Z$ such that
\[
X_1=\mathcal S_1^B\oplus W.
\]
Note that $\dim W \le \dim Z\le d$. We conclude that $\mathcal S_1^B$ is complemented and thus by Theorem~\ref{Hthm2} we see that~\eqref{B11} admits an exponential dichotomy with respect to some projections $\tilde P_{11}(n)$, $n\in \Z^+$. Then, \eqref{iid} admits an exponential dichotomy for $i=1$ and with respect to projections $P_{11}(n)$ given by $P_{11}(n)=\tilde P_{11}(-n)^*$, $n\in \Z^-$. Moreover, $\dim \mathcal NP_{11}(0)=\dim \mathcal N\tilde P_{11}(0)$. Hence,
\[
d_1(m)=\dim \mathcal U_1(m)=\dim \mathcal NP_{11}(m)=\dim \mathcal NP_{11}(0)=\dim \mathcal N\tilde P_{11}(0)\le d,
\]
for each $m\in \Z^-$.

$(2)$ Assume first that~\eqref{dddeq2} holds. In particular, we have $\dim \mathcal U_2(0)<+\infty$, which implies that $\mathcal U_2(0)$ is closed and complemented. 

Moreover, since $d_1(m)$ does not depend on $m$, we find that $d_2(m)$ is also independent of $m$. On the other hand, it is easy to verify (using the same arguments as in the proof of~\cite[Lemma 2]{DP}) that $A_{22}(m)\mathcal U_2(m)= U_2(m+1)$ for $m\le -1$, and thus $A_{22}(m)\rvert_{\mathcal U_2(m)}\colon \mathcal U_2(m)\to \mathcal U_2(m+1)$ is invertible.

Suppose now that there is a nonzero bounded solution $(x_2(n))_{n\in J}$ of~\eqref{iid} for $i=2$ with $x_2(0)=0$. Choose the largest $\ell \in J'$ such that $x_2(\ell)\neq 0$. Then, $0=x_2(\ell+1)=A_{22}(\ell)x_2(\ell)$. Since $x_2(\ell)\in \mathcal U_2(\ell)$ and $A_{22}(\ell)\rvert_{\mathcal U_2(\ell)}$ is injective, we obtain a contradiction.  From the above observations and the Lemma~\ref{neglin}, we conclude that~\eqref{iid} admits an exponential dichotomy for $i=2$, which together with the first assertion of the theorem yields the existence of an exponential dichotomy for~\eqref{ldedd}.

Conversely, we now assume that~\eqref{ldedd} admits an exponential dichotomy. Hence, \eqref{iid} admits an exponential dichotomy with respect to projections $P_{ii}(n)$, $n\in J$ for $i=1, 2$. Then, $\mathcal NP_{ii}(m)=\mathcal U_i(m)$ for $m\in J$ and $i=1, 2$. Let $\mathcal U_2'(m)$ denote the subspace consisting of all $\eta \in \mathcal U_2(m)$ for which there exists a bounded solution $(x_1(n))_{n\le m}\subset X_1$ of the equation
\[
x_1(n+1)=A_{11}(n)x_1(n)+A_{12}(n)x_2^\eta (n) \quad n<m, 
\]
where $(x_2^\eta(n))_{n\le m}\subset X_2$ is the unique bounded sequence such that $x_2^\eta(m)=\eta$ and $x_2^\eta(n)=A_{22}(n-1)x_2^\eta(n-1)$ for $n\le m$. The above equation has a unique solution with $x_1(m)\in \mathcal R P_{11}(m)$. Let $L\eta:=x_1(m)$. Then, $L\colon \mathcal U_2'(m)\to \mathcal RP_{11}(m)$ is a linear operator and 
\[
\mathcal N P(m)=\{(\xi+L\eta, \eta): \ \xi \in \mathcal NP_{11}(m), \ \eta \in \mathcal U_2'(m)\}.
\]
Therefore, 
\[
\begin{split}
d=\dim \mathcal NP(0)=\dim \mathcal NP(m) &= \dim  \mathcal NP_{11}(m)+\dim \mathcal U_2'(m)\\
&=d_1(m)+\dim \mathcal U_2'(m) \\
&\le d_1(m)+d_2(m),
\end{split}
\]
and consequently
\[
d\le d_1(m)+d_2(m), \quad m\in J.
\]
By relying on the exponential dichotomy of~\eqref{eq851}, one can easily show that 
\[
d_1(m)+d_2(m)\le d \quad m\in J,
\]
yielding~\eqref{dddeq2}.
\end{proof}


\subsection{Dichotomies on $\Z^+$ and $\Z^-$: systems with invertible coefficients}

In this subsection, we deal with the invertible systems on $\Z^+$ and $\Z^-$. In sharp contrast to the finite-dimensional case, in Example~\ref{EX2} we have shown that the existence of an exponential dichotomy of~\eqref{LDE} does not imply the existence of an exponential dichotomy of~\eqref{ldedd} for invertible systems on $J\in \{\Z^+, \Z^-\}$ (even in the autonomous case). However, we are able to formulate some positive results in this direction.
\begin{theorem}\label{504thm}
Let $J=\mathbb Z^+$ and assume that $A_{ii}(n)$ is an invertible operator for $i=1, 2$ and $n\in J$. Suppose that~\eqref{LDE} admits an exponential dichotomy. Then~\eqref{ldedd} admits an exponential dichotomy if and only if  the following holds:
\begin{enumerate}
\item $\mathcal S_2$ given by~\eqref{Si} for $i=2$ is complemented in $X_2$;
\item the subspace 
\begin{equation}\label{subspace}
\left \{ v_1\in X_1: \ \sup_{n\in J}\|(\Phi_{11}(n, 0)^*)^{-1}v_1\|<+\infty \right \}
\end{equation}
is complemented in $X_1$.
\end{enumerate}
\end{theorem}

\begin{proof}
Suppose that $\mathcal S_2$ is complemented in $X_2$ and that the subspace in~\eqref{subspace} is complemented in $X_1$.
It follows from our assumptions and Lemma~\ref{diciid2} that~\eqref{iid} admits an exponential dichotomy for $i=2$. 
Since~\eqref{LDE} admits an exponential dichotomy, we have (by the same reasoning as in the proof of~\cite[Lemma 1]{P}) that  
\begin{equation}\label{star}
x(n+1)=(A(n)^*)^{-1} x(n) \quad n\in J
\end{equation}
admits an exponential dichotomy
In addition, 
\begin{equation}\label{501coeff}
(A(n)^*)^{-1} =\begin{pmatrix}
(A_{11}(n)^*)^{-1} & 0\\
B_{21}(n) & (A_{22}(n)^*)^{-1}
\end{pmatrix} \quad n\in J,
\end{equation}
where $B_{21}(n)\in \mathcal L(X_1, X_2)$.
Now it is easy to see that the system
\begin{equation}\label{kuj}
x_1(n+1)=(A_{11}(n)^*)^{-1} x_1(n) \quad n\in J
\end{equation}
admits an exponential dichotomy. To this end, we take a bounded sequence $(y_1(n))_{n\in J}\subset X_1$. Set $y(n)=(y_1(n), 0)\in X_1\times X_2$ for $n\in J$. Then $(y(n))_{n\in J}$ is bounded. Since~\eqref{star} admits an exponential dichotomy, there exists a bounded sequence $(x(n))_{n\in J}\subset X_1\times X_2$ such that 
\[
x(n+1)=(A(n)^*)^{-1}x(n)+y (n), \quad n\in J.
\]
Writing $x(n)=(x_1(n), x_2(n))$ with $x_i(n)\in X_i$ for $i=1, 2$, we have 
\[
x_1(n+1)=(A_{11}(n)^*)^{-1}x_1(n)+y_1(n), \quad n\in J.
\]
From the above and the assumption that the subspace in~\eqref{subspace} is complemented in $X_1$, we conclude (via Theorem~\ref{Hthm}) that~\eqref{kuj} admits an exponential dichotomy. Hence, \eqref{iid} admits an exponential dichotomy for $i=1$.

The converse direction is trivial.
\end{proof}

\begin{theorem}\label{515thm}
Let $J=\mathbb Z^+$ and assume that $A_{ii}(n)$ is an invertible operator for $i=1, 2$ and $n\in J$. Suppose that~\eqref{LDE} admits an exponential dichotomy with respect to projections $P(n)$, $n\in J$, and that $d:=\dim \mathcal NP(0)<+\infty$. Then \eqref{ldedd} admits an exponential dichotomy.

\end{theorem}

\begin{proof}
It follows from Theorem~\ref{thm7} that~\eqref{iid} admits an exponential dichotomy for $i=2$. Moreover, $\mathcal S_1$ (given by~\eqref{Si} for $i=1)$ is closed. 
Next, we have
\[
\mathcal R P(0)=\{(\xi+L\eta, \eta): \xi \in \mathcal S_1, \ \eta \in \mathcal S_2'\},
\]
where $\mathcal S_2'$ is as in the proof of Lemma~\ref{ddd} and $L\colon \mathcal S_2'\to \mathcal S_1^\perp$ is a linear operator given by $L\eta=x_1(0)$, where $(x_1(n))_{n\in J}\subset X_1$ is the unique bounded solution of the equation
\[
x_1(n+1)=A_{11}(n)x_1(n)+A_{12}(n)\Phi_{22}(n, 0)\eta \quad n\in J,
\]
with $x_1(0)\in \mathcal S_1^\perp$. Note that 
\[
X_1\times X_2=\mathcal RP(0)\oplus (\mathcal S_1^\perp \times W),
\]
where $W$ is any subspace of $X_2$ such that $X_2=\mathcal S_2'\oplus W$. Hence, $\dim \mathcal S_1^\perp<+\infty$.
\begin{lemma}\label{bounded}
$L$ is a bounded linear operator.
\end{lemma}

\begin{proof}[Proof of the lemma]
Since $\dim \mathcal S_1^\perp<+\infty$, it suffices to show that $\mathcal N L$ is closed. To this end, we take a sequence $(\eta_k)_k \in \mathcal N L$ that converges to $\eta$ in $X_2$. Then $(0, \eta_k)_k$ is a sequence in $\mathcal RP(0)$ that converges to $(0, \eta)$. Since $\mathcal RP(0)$ is closed, we have $(0, \eta)\in \mathcal R P(0)$. In particular, $\eta \in \mathcal S_2'$ and $L\eta=0$
Hence, $\eta \in \mathcal N L$.
\end{proof}
\begin{lemma}\label{closed}
$\mathcal S_2'$ is closed.
\end{lemma}

\begin{proof}[Proof of the lemma]
Take a sequence $(\eta_k)_k$ in $\mathcal S_2'$ that converges to $\eta\in X_2$. Since $\mathcal S_2$ is closed, we have $\eta \in \mathcal S_2$.

By $(x_1^k(n))_{n\in J}$ we will denote the bounded sequence in $X_1$ such that 
\begin{equation}\label{x11k}
    x_1^k(n+1)=A_{11}(n)x_1^k(n)+A_{12}(n)\Phi_{22}(n, 0)\eta_k \quad n\in J,
\end{equation}
and $x_1^k(0)=L\eta_k \in \mathcal S_1^\perp$. In particular, for each $k\in \mathbb N$,
\[
(x_1^k(n), \Phi_{22}(n, 0)\eta_k)=\Phi(n, 0)(L\eta_k, \eta_k), \quad n\in J.
\]
Since $(L\eta_k, \eta_k)\in \mathcal R P(0)$, we have
\[
\|x_1^k(n)\| \le \|\Phi(n, 0)(L\eta_k, \eta_k)\| \le Ce^{-\lambda n} \|(L\eta_k, \eta_k)\| \le C(1+\|L\|)\sup_{k\in \mathbb N}\|\eta_k\|,
\]
for $n\in J$ and $k\in \mathbb N$, where $C, \lambda>0$ are constants associated with the exponential dichotomy of~\eqref{LDE}. Consequently, 
\[
\|x_1^k(n)\| \le D \quad \text{for $n\in J$ and $k\in \mathbb N$,}
\]
where $D:=C(1+\|L\|)\sup_{k\in \mathbb N}\|\eta_k\|$.

Since $X_1$ is reflexive (as it is a Hilbert space), for each $n\in J$, the sequence $(x_1^k(n))_k$ has a weakly convergent subsequence whose limit belongs to the closed ball in $X_1$ with radius $D$ centered in $0$. Using the diagonal procedure, we can find a subsequence $(k_j)_j$ of $\mathbb N$ such that $(x_1^{k_j}(n))_j$ converges weakly to $x_1(n)\in X_1$ with $\|x_1(n)\| \le D$, for every $n\in J$.
Applying~\eqref{x11k} for $k=k_j$, passing to the limit when $j\to \infty$, and recalling that weak convergence is preserved under the action of a bounded linear operator, we get
\[
x_1(n+1)=A_{11}(n)x_1(n)+A_{12}(n)\Phi_{22}(n, 0)\eta \quad n\in J.
\]
Therefore, $\eta\in \mathcal S_2'$. 
\end{proof}
Let $P_{11} \colon X_1\to \mathcal S_1$ be the projection on $\mathcal S_1$ along $\mathcal S_1^\perp$, and $P_{22}\colon X_2 \to \mathcal S_2'$ be the projection on $\mathcal S_2'$ along $(\mathcal S_2')^\perp$. We define the projection $\tilde P$ on $X_1\times X_2$ by 
\begin{equation}\label{tildeP}
\tilde P=\begin{pmatrix}
P_{11} & LP_{22}\\
0 &P_{22}
\end{pmatrix}.
\end{equation}
From the previous two lemmas it follows that $\tilde P$ is bounded. Moreover, $\mathcal R\tilde P=\mathcal RP(0)$. Set
\[
\tilde P(n):= \Phi(n, 0)\tilde P\Phi(0, n), \quad n\in J.
\]
Note that 
\[
\tilde P(n)=\begin{pmatrix}
\Phi_{11}(n, 0)P_{11}\Phi_{11}(0, n) & \star\\
0 &\Phi_{22}(n, 0)P_{22}\Phi_{22}(0, n)
\end{pmatrix}, \quad n\in J.
\]
Then \eqref{LDE} admits an exponential dichotomy with respect to the sequence of projections $\tilde P(n)$, $n\in J$. This easily implies that~\eqref{iid} for $i=2$ admits an exponential dichotomy with respect to the sequence of projections $\tilde P_{22}(n)$, $n\in J$ given by 
\[
\tilde P_{22}(n)=\Phi_{22}(n, 0)P_{22}\Phi_{22}(0, n), \quad n\in J.\]
\end{proof}

\begin{remark}\label{Remark8}
In the proof of Theorem~\ref{515thm} it was necessary to show that~\eqref{LDE} admits an exponential dichotomy with respect to a sequence of bounded projections that have an upper triangular form. For this purpose, we had to prove that $\mathcal S_2'$ is closed and that $L$ is a bounded linear operator, so that $\tilde P$ from~\eqref{tildeP} is a bounded projection. This again illustrates substantial difficulties arising by working in an infinite-dimensional setting.
\end{remark}

Before stating the next result, we observe that in the case when operators $A_{ii}(n)$, $n\in J$ are invertible, we can set
\[
\Phi_{ii}(m, n):= (\Phi_{ii}(n, m))^{-1}, \quad m, n\in J, \ m \le n.
\]
\begin{theorem}\label{456thm}
Let $J=\mathbb Z^-$ and assume that $A_{ii}(n)$ is an invertible operator for $i=1, 2$ and $n\in J$. Suppose that~\eqref{LDE} admits an exponential dichotomy. Then~\eqref{ldedd} admits an exponential dichotomy if and only if  the following holds:
\begin{enumerate}
 \item the subspace
 \[
 \mathcal U_2=\left \{v_2\in X_2: \ \sup_{n\in J}\|\Phi_{22}(n, 0)v_2\|<+\infty \right \}
 \]
 is complemented in $X_2$;
 
\item the subspace
 \begin{equation}\label{458subs}
 \left \{v_1\in X_1: \ \sup_{n\in J}\|\Phi_{11}(0, n)^*v_1\|<+\infty \right \}
 \end{equation}
  is complemented in $X_1$.
\end{enumerate}
\end{theorem}

\begin{remark}\label{459rem}
We note that $\mathcal U_2$ in the statement of Theorem~\ref{456thm} coincides with $\mathcal U_2$ introduced in the statement of Lemma~\ref{neglin}.
\end{remark}

\begin{proof}
Suppose that $\mathcal U_2$ is complemented in $X_2$ and that the subspace in~\eqref{458subs} is complemented in $X_1$. From Remark~\ref{459rem} and Lemma~\ref{neglin}, we have that~\eqref{iid} admits an exponential dichotomy for $i=2$.

Since~\eqref{LDE} admits an exponential dichotomy, we find that~\eqref{star} admits an exponential dichotomy. Moreover, the coefficients of~\eqref{star} are of the form~\eqref{501coeff}. By proceeding as in the proof of Theorem~\ref{504thm}, we find that for each bounded sequence $(y_1(n))_{n\in J}\subset X_1$, there exists a bounded sequence $(x_1(n))_{n\in J}$ such that 
\[
x_1(n+1)=(A_{11}(n)^*)^{-1}x_1(n)+y_1(n), \quad \text{for $n\in J'$.}
\]
In addition, the subspace of $X_1$ consisting of all $v_1\in X_1$ for which there is a bounded sequence $(x_1(n))_{n\in J}$ with $x_1(0)=v_1$ and $x_1(n)=(A_{11}(n-1)^*)^{-1}x_1(n-1)$ for $n\in J$ coincides with the one introduced in~\eqref{458subs}, which together with Theorem~\ref{Hthm3} implies that~\eqref{kuj} admits an exponential dichotomy. Therefore, \eqref{iid} for $i=1$ admits an exponential dichotomy. 

The opposite direction is straightforward. 
\end{proof}

\begin{theorem}
Let $J=\mathbb Z^-$ and assume that $A_{ii}(n)$ is an invertible operator for $i=1, 2$ and $n\in J$. Suppose that~\eqref{LDE} admits an exponential dichotomy with respect to projections $P(n)$, $n\in J$, and that $d:=\dim \mathcal NP(0)<+\infty$. Then \eqref{ldedd} admits an exponential dichotomy.

\end{theorem}

\begin{proof}
The proof is similar to the proof of Theorem~\ref{515thm}, and therefore we only provide a sketch of it. 
Firstly, from Theorem~\ref{517thm} it follows that~\eqref{iid} admits an exponential dichotomy for $i=1$.

Next, we consider the system~\eqref{eq851} with $B(n)$ as in~\eqref{555eq} for $n\in \Z^+$. Since~\eqref{LDE} admits an exponential dichotomy, we find that~\eqref{eq851} admits an exponential dichotomy with respect to projections $\tilde P(n)=P(-n)^*$, $n\in \Z^+$. Moreover, $\dim \mathcal N\tilde P(0)=\dim \mathcal NP(0)=d$. Using the same notation as in the proof of Theorem~\ref{517thm},  we see that
\[
\mathcal R\tilde P(0)=\{(\xi, L\xi+\eta): \ \xi \in (\mathcal S_1^B)', \ \eta \in \mathcal S_2^B\},
\]
where
\[
\mathcal S_2^B=\left \{v_2\in X_2: \ \sup_{n\in \Z^+}\|\Phi_{22}^B(n, 0)v_2\|<\infty\right \},
\]
with $\Phi_{22}^B(m, n)$ denoting cocycle associated with the discrete system
\begin{equation}\label{654s}
x_{22}(n+1)=B_{22}(n)x_{22}(n), \quad n\in \Z^+.
\end{equation}
Moreover, $(\mathcal S_1^B)'$ consists of all $\xi\in \mathcal S_1^B$ for which there exists a bounded solution $(x_2(n))_{n\in \Z^+}$ of the equation
\[
x_2(n+1)=B_{21}(n)\Phi_{11}^B(n, 0)\xi+B_{22}(n)x_2(n), \quad n\in \Z^+.
\] Observe that $\mathcal S_2^B$ is closed, since $v_2\in \mathcal S_2^B$ implies $(0, v_2)\in \mathcal R\tilde P(0)$.
Finally, $L\xi=x_2(0)$ where $(x_2(n))_{n\in \Z^+}$ is the unique solution of the above equation with $x_2(0)\in (\mathcal S_2^B)^\perp$. Since $d<+\infty$, we have $\dim (\mathcal S_2^B)^\perp<+\infty$.

Hence, arguing as in the proof of Lemma~\ref{bounded} one can show that $L\colon (\mathcal S_1^B)'\to (\mathcal S_2^B)^\perp$ is a bounded linear operator. Note that $\mathcal S_1^B$ is closed since~\eqref{iid}  for $i=1$ (and thus also~\eqref{B11}) admits an exponential dichotomy.
Hence, proceeding as in the proof of Lemma~\ref{closed}, we find that $(\mathcal S_1^B)'$ is closed.

Let $P_{11}\colon X_1\to (\mathcal S_1^B)'$ be the projection onto $(\mathcal S_1^B)'$ along $((\mathcal S_1^B)')^\perp$, and $P_{22}\colon X_2\to \mathcal S_2^B$ the projection onto $\mathcal S_2^B$ along $(\mathcal S_2^B)^\perp$. We define a projection $\bar P$ on $X_1\times X_2$ by 
\[
\bar P=\begin{pmatrix}
P_{11} & 0\\
LP_{11} &P_{22}
\end{pmatrix}.
\]
Then, $\bar P$ is bounded and $\mathcal R\bar P=\mathcal R\tilde P(0)$. Consequently, \eqref{eq851} admits an exponential dichotomy with respect to the sequence of projections
\[
\bar P(n)=\Phi^B(n, 0)\bar P\Phi^B(0, n) \quad n\in \Z^+,
\]
where $\Phi^B(m, n)$ is the cocycle associated with~\eqref{eq851}. This easily implies that~\eqref{654s} admits an exponential dichotomy with respect to the sequence of projections
\[
P_{22}(n)=\Phi_{22}^B(n, 0)P_{22}\Phi_{22}^B(0, n), \quad n\in \Z^+.
\]
We conclude that~\eqref{iid} admits an exponential dichotomy for $i=2$, which completes the proof of the theorem.

\end{proof}

\section*{Acknowledgements}  
The authors thank N. Chalmoukis and the MathOverflow community for providing Example~\ref{EX2}.
D. Dragi\v cevi\' c was supported in part by the University of Rijeka under the project uniri-iz-25-108.

\end{document}